\documentclass{article}
\usepackage{arxiv}
\usepackage{lmodern}
\usepackage{amsmath, amsfonts, amssymb, geometry, color, pgf, tikz, graphicx, yhmath, bm, esvect, framed, autobreak, tabularray, amsthm}

\newcommand{\pa}{\ \, /\kern -0.75em / \ \,} 

\makeatletter

\newcommand{\Rmnum}[1]{\expandafter\@slowromancap\romannumeral #1@}
\makeatother

\newtheorem{theorem}{Theorem}

\begin{document}

\title{Metric Matrix in Barycentric Coordinates}
\markright{Submission}
\author{Xi Feng}
\maketitle

\begin{abstract}
In this paper, the concept of the metric matrix is introduced to establish a concise and unified formulation for the inner product in barycentric coordinates. Building on this framework, we explore the intrinsic algebraic identities of barycentric coordinates and their direct correspondence with geometric theorems. Through this approach, we derive a series of novel results and provide new proofs for several classical theorems in triangle geometry. 
\end{abstract}
\keywords{Barycentric Coordinates \and Metric matrix \and Triangle geometry}
\section{Introduction}
\noindent
Barycentric coordinates provide an effective and natural framework for the study of triangle geometry, owing to their precise characterization of points and the inherent symmetry they preserve within the triangle. Fundamental formulas and computational techniques have been extensively discussed in the literature (\cite{capitan2015barycentric}; \cite{grozdev2016barycentric};
\cite{schindler2012barycentric}; 
\cite{yiu2001introduction};), and the explicit barycentric coordinates of numerous triangle centers are cataloged in the Encyclopedia of Triangle Centers (\cite{kimberling_etc}, hereafter referred to as ETC). Nevertheless, the representation of angles within this coordinate system is often algebraically intricate, which renders calculations particularly cumbersome when addressing problems involving circles. More generally, computations carried out in barycentric coordinates tend to be lengthy and offer limited geometric insight.

In this work, we revisit the foundational results of planar geometry from the Cartesian coordinate perspective and establish their barycentric counterparts, with particular emphasis on deriving the inner product formula between two points. Although similar expressions have been reported in prior studies (\cite{volenec2003metrical}; \cite{volenec2015baricentricke}; \cite{zolotov2022scalar}), we introduce the concept of the metric matrix to establish the inner product formula with a compact matrix representation. This approach also yields an elegant and concise formulation for angle calculations.

Building upon this theoretical foundation, several classical theorems, as collected in \cite{johnson2013advanced}, are revisited and re-established through this novel and streamlined methodology.

Our investigation reveals that many geometric theorems can be reformulated as matrix identities in barycentric coordinates, and that the transformation of these identities can often supplant traditional geometric derivations, leading to new and insightful results. Moreover, while the introduction of new points typically involves inner and cross product operations on existing points, their explicit coordinate expressions are frequently unnecessary for the derivation of geometric properties.
\section{Basic formulae in barycentric coordinate}
\label{sec:bary}
Let $\triangle ABC$ be a fixed nondegenerate reference triangle. The side lengths of the $\triangle ABC$ are denoted by $a = |BC|,\,\,b = |CA|$ and $c = |AB|$, where $|XY|$ denotes the Euclidean distance between points $X$ and $Y$. The Conway notation is introduced as follows: 
\begin{align}
    &S_A=\frac 12\left(b^2+c^2-a^2\right),
   &&S_B=\frac 12\left(c^2+a^2-b^2\right),
   &&S_C=\frac 12\left(a^2+b^2-c^2\right).
\end{align}
For a point $P$ in the plane, an ordered triple $(\alpha_P,\, \beta_P,\, \gamma_P)$ is introduced to represent the position $P$ relative to the reference triangle $\triangle ABC$, such that the following relationship holds:
\begin{align}
    x_A \alpha_P+x_B \beta_P+x_C \gamma_P&=x_P,\notag\\
    y_A \alpha_P+y_B \beta_P+y_C \gamma_P&=y_P,
    \label{tranform}\\
    \alpha_P+\beta_P+\gamma_P&=1,\notag
\end{align}
where $(x_P, y_P)$ denotes the Cartesian coordinates of the point $P$. The equations (\ref{tranform}) can be converted to the matrix form as follows:
\begin{align}
    \mathbf{S}\mathbf{P}^T=
  \begin{bmatrix}      x_P \\ y_P \\ 1\end{bmatrix},
  \label{transform formulae}
\end{align}
where
\begin{align*}
    P=\begin{bmatrix}\alpha_P & \beta_P & \gamma_P\end{bmatrix},
\end{align*}
and
\begin{align*}
    \mathbf{S}=\begin{bmatrix}
    x_A & x_B & x_C \\
    y_A & y_B & y_C \\
    1 & 1   & 1   \\
    \end{bmatrix}.
\end{align*}
Without causing confusion, non-bold capitals represent both points and their barycentric coordinate row vectors in the following article. The equation (\ref{transform formulae}) is referred to as the coordinate transformation formula, 
and the oriented area of the reference triangle $ABC$ (the counterclockwise arrangement $A,\, B, \,C$ is positive) is equal to one-half of the determinant of the matrix $\mathbf{S}$, i.e.,
\begin{align*}
    S_{\triangle ABC}=\frac12\begin{vmatrix}
    x_A & x_B & x_C \\
    y_A & y_B & y_C \\
    1 & 1   & 1   \\
    \end{vmatrix}.
\end{align*}
Using Cramer's rule, we can obtain
\begin{align*}
    &\alpha_P =\frac{1}{2S_{\triangle ABC}}\begin{vmatrix}
    x_P & x_B & x_C \\
    y_P & y_B & y_C \\
    1 & 1   & 1   \\
    \end{vmatrix},&&
    \beta_P =\frac{1}{2S_{\triangle ABC}}\begin{vmatrix}
    x_A & x_P & x_C \\
    y_A & y_P & y_C \\
    1 & 1   & 1   \\
    \end{vmatrix},&&
    \gamma_P =\frac{1}{2S_{\triangle ABC}}\begin{vmatrix}
    x_A & x_B & x_P \\
    y_A & y_B & y_P \\
    1 & 1   & 1   \\
    \end{vmatrix}.
\end{align*}
So $P$ can be written as
\begin{align}
    P = \frac{1}{S_{\triangle ABC}}
\begin{bmatrix}      S_{\triangle PBC} & S_{\triangle APC} & S_{\triangle ABP}\end{bmatrix}.
\label{barycentric}
\end{align}
The expression (\ref{barycentric}) defines the normalized barycentric coordinates of point $P$, where the sum of its three components equals $1$. 
In addition to the standard points of Euclidean plane, it also includes a point at infinity for each class of parallel lines. Both standard points and points at infinity construct the real projective plane.
A point at infinity can be expressed as the difference between the coordinates of two points, $P-Q$, with zero sum of coordinates.

For three points $P,\, Q,\, R$ in the plane, using the coordinate transformation formula, we have
\begin{align*}
    \mathbf{S}\left[P;Q;R\right]^T=
  \begin{bmatrix}      x_P&x_Q&x_R \\ y_P&y_Q&y_R \\ 1&1&1\end{bmatrix}.
\end{align*}
Using the determinant multiplicative property, we have
\begin{align}
\left[P;Q;R\right]=\frac{S_{\triangle PQR}}{S_{\triangle ABC}},
\label{area formula}
\end{align}
where
\begin{align*}
\left[P;Q;R\right]=\begin{vmatrix}
    \alpha_P & \alpha_Q & \alpha_Q \\
    \beta_P & \beta_Q & \beta_Q \\
    \gamma_P & \gamma_Q   &\gamma_Q  \\
\end{vmatrix}.
\end{align*}
In particular, if three points $P,\, Q,\, R$ are collinear, then
\begin{align*}
    \left[P;Q;R\right]=0.
\end{align*}
In the Cartesian coordinates, the formula for the inner product of vectors can be written as
\begin{align*}
    \vv{PQ}\cdot \vv{RT}
=\left(\begin{bmatrix}x_Q & y_Q & 1\end{bmatrix}-\begin{bmatrix}x_P & y_P & 1\end{bmatrix}\right)
\left(\begin{bmatrix}x_T \\ y_T \\ 1\end{bmatrix}-\begin{bmatrix}x_R \\ y_R \\ 1\end{bmatrix}\right),
\end{align*}
where $P,\,Q,\,R,\,T$ denote four points in the plane.
Using the transformation formula (\ref{transform formulae}), we obtain
\begin{align*}
   \left(\begin{bmatrix}x_Q & y_Q & 1\end{bmatrix}-\begin{bmatrix}x_P & y_P & 1\end{bmatrix}\right)
\left(\begin{bmatrix}x_T \\ y_T \\ 1\end{bmatrix}-\begin{bmatrix}x_R \\ y_R \\ 1\end{bmatrix}\right)=
\left(Q-P\right)
\mathbf{S}^T\mathbf{S}
\left(T-R\right)^T,
\end{align*}
where
\begin{align*}
    \mathbf{S}^T\mathbf{S}=\begin{bmatrix}
x_A^2+y_A^2+1 & x_Ax_B+y_Ay_B+1 & x_Ax_C+y_Ay_C+1 \\ x_Ax_B+y_Ay_B+1 & x_B^2+y_B^2+1 & x_Bx_C+y_By_C+1 \\ 
x_Ax_C+y_Ay_C+1 & x_Bx_C+y_By_C+1 & x_C^2+y_C^2+1 \\ \end{bmatrix}.
\end{align*}
$\mathbf{S}^T\mathbf{S}$ is referred to as the metric matrix. From equation (\ref{barycentric}),
\begin{align*}
    \begin{bmatrix}1 & 1 & 1\end{bmatrix}P^T=1.
\end{align*}
Therefore, for any real numbers $m,\,n,\,l$, we have
\begin{align*}
&\left(Q-P\right)
\begin{bmatrix}m \\  n \\ l\end{bmatrix}
\begin{bmatrix}1 & 1 & 1\end{bmatrix}
\left(T-R\right)^T=0,\\
&\left(Q-P\right)
\begin{bmatrix}1 \\  1 \\ 1\end{bmatrix}
\begin{bmatrix}m & n & l\end{bmatrix}
\left(T-R\right)^T=0.
\end{align*}
Let
\begin{align*}
    f\left(m,\,n,\,l\right)=\frac{1}{2}
    \begin{bmatrix}m \\  n \\ l\end{bmatrix}
    \begin{bmatrix}1 & 1 & 1\end{bmatrix}+
    \frac{1}{2}
\begin{bmatrix}1 \\  1 \\ 1\end{bmatrix}
\begin{bmatrix}m & n & l\end{bmatrix}
= \frac{1}{2}
\begin{bmatrix}2m & m+n & m+l\\
               m+n& 2n  & n+l \\
               m+l& n+l & 2l\end{bmatrix},
\end{align*}
Therefore, we have
\begin{align*}
    \left(Q-P\right)
f\left(m,\,n,\,l\right)
\left(T-R\right)^T=0.
\end{align*}
Therefore, if $K$ is a metric matrix, then $K \pm f(m,n,l)$ is still a metric matrix. Without causing confusion, $K$ is used to represent the metric matrix in the following, and the key result could be summarized as follows:
\begin{theorem}[Inner product formula]
    \begin{align}
      \vv{PQ}\cdot \vv{RT}&=\left(Q-P\right)K\left(T-R\right)^T,
      \label{inner product barycentric}
      \end{align}
\end{theorem}
As a particular case, the formula for the distance between two points takes the form
\begin{align}
    |PQ|^2=\left(P-Q\right)K\left(P-Q\right)^T.
    \label{distance formula barycentric}
\end{align}
And the condition for two lines to be perpendicular can be obtained as follows:
\begin{align}
    PQ\perp RT\iff \left(Q-P\right)K\left(T-R\right)^T=0.
    \label{perpendicular condition}
\end{align}
Simple calculations yield
\begin{align*}
-\frac 12
    \begin{bmatrix}
        0 &c^2 & b^2\\
        c^2&0&a^2\\
        b^2&a^2&0
    \end{bmatrix}=\mathbf{S}^T\mathbf{S}
    -f\left(1,\,1,\,1\right)
    -f\left(|OA|^2,\,|OB|^2,\,|OC|^2\right)
\end{align*}
The obtained matrix is denoted as $K_O$, which is obviously only related to the size and shape of the triangle and is independent of the choice of the Cartesian coordinate system! Another important metric matrix is
\begin{align}
    K_H=\begin{bmatrix}
S_A & 0  & 0 \\
0  & S_B & 0 \\
0  & 0  & S_C \\
\end{bmatrix},
\label{expressions KH}
\end{align}
where
\begin{align*}
K_H=
    K_O
    +f\left(S_A,\,S_B,\,S_C\right).
\end{align*}
In the following article, $K_H$ is used as a representative of the metric matrix and is abbreviated as $K$.

\noindent In order to provide a clearer description of the vector cross product, the notion of an oriented angle is introduced. Let $\measuredangle QPR$ denote the directed angle (with the positive direction being clockwise for the arrangement $Q,\,P,\,R$); we have
\begin{align*}
    \vv{PQ}\times \vv{PR}&=|PQ|\cdot|PR|\sin\measuredangle QPR
    =2S_{\triangle PQR}.
\end{align*}
It is worth noting that the direction of the cross product has been ignored. Using (\ref{area formula}), we have
\begin{align*}
    \vv{PQ}\times\vv{PR}=2S_{\triangle ABC}\left[P;Q;R\right].
\end{align*}
Therefore, general cross product formula can be written as follows:
\begin{align}
    \vv{PQ}\times\vv{RT}
    &=\vv{PQ}\times\vv{PT}-\vv{PQ}\times\vv{PR}=2S_{\triangle ABC}\left[P;Q;T-R\right].
    \label{cross product formula}
\end{align}
Obverving that
\begin{align*}
    \vv{PQ}\cdot \vv{PR}&=|\vv{PQ}|\cdot|\vv{PR}|
    \cos\measuredangle QPR=\left(P-Q\right)K\left(P-R\right)^T,
\end{align*}
thus an important angle calculation formula could be obtained as follows:
\begin{theorem}[Oriented angle formula]
        \begin{align}
        \cot\measuredangle QPR=\frac{\left(P-Q\right)K\left(P-R\right)^T}{2S_{\triangle ABC}\left[P;Q;R\right]}.
        \label{angle formula barycentric}
    \end{align}
\end{theorem}
\noindent At this point, the inner product formula (\ref{inner product barycentric}) and cross product formula (\ref{cross product formula}), along with several important corollaries ((\ref{distance formula barycentric}), (\ref{perpendicular condition}) and (\ref{angle formula barycentric})) in barycentric coordinates, have been obtained. In the next section, basic properties of the reference triangle $\triangle ABC$ are introduced utilizing the formulas obtained in this section.
\section{Basic properties of the reference triangle}
\label{sec:triangle}
\noindent Let $X=\begin{bmatrix}x & y & z\end{bmatrix}$ be an arbitrary point in the plane; according to the equation (\ref{barycentric}), we have
\begin{align*}
    &x = \left[B;C;P\right],
   &&y = \left[C;A;P\right],
   &&z = \left[A;B;P\right].
\end{align*}
Therefore, the vertices of the triangle $\triangle ABC$ have simple coordinates:
\begin{align*}
    &A=\begin{bmatrix}1 & 0 & 0\end{bmatrix},
    &&B=\begin{bmatrix}0 & 1 & 0\end{bmatrix},
    &&C=\begin{bmatrix}0 & 0 & 1\end{bmatrix}.
\end{align*}
As the metric matrix $\mathbf{K_H}$ is abbreviated by $K$, simple calculations yield
\begin{align*}
    &AKA^T=S_A,&&BKB^T=S_B,&&CKC^T=S_C,\\
    &AKB^T=0,&&BKC^T=0,&&CKA^T=0.
\end{align*}
Therefore, the equations of the lines containing the sides $BC,\,CA,\,AB$ are, respectively,
\begin{align*}
    &AKX^T=0,
    &&BKX^T=0,
    &&CKX^T=0.
\end{align*}
The Conway notation $S_A,\,S_B,\,S_C$ involved in the expression of $K$ is usually used to represent the fundamental quantities of the triangle $\triangle ABC$. Let $a,\,b,\,c$ be the sides opposite to angles $\angle A,\,\angle B,\,\angle C$, respectively, then
\begin{align*}
    &a^2=S_B+S_C,
   &&b^2=S_C+S_A,
   &&c^2=S_A+S_B.
\end{align*}
The area $S$ of the triangle $\triangle ABC$ can be expressed as
\begin{align*}
S=\frac 12\left(S_BS_C+S_CS_A+S_AS_B\right)^{1/2},
\end{align*}
which is equivalent to Heron's formula.

\noindent $G$ is denoted as
\begin{align*}
    G=\frac13\left(A+B+C\right),
\end{align*}
rearranging gives
\begin{align}
    3G-A=2\cdot\frac 12\left(B+C\right).
    \label{G property}
\end{align}
Clearly, $G$ lies on the median through point $A$. By symmetry, the three medians of the triangle $\triangle ABC$ intersect at the point $G$, so it is the centroid of the triangle $\triangle ABC$, denoted
as $X(2)$ in ETC.
(\ref{G property}) also indicates that the centroid divides the median with the ratio of $1:2$. 

\noindent A classical extremal property of the centroid is concluded as follows:
\begin{theorem}
    The centroid of a triangle minimizes the sum of the squares of the distances to the three vertices.
\end{theorem}
\begin{proof}
Since
\begin{align*}
    9GKG^T=AKA^T+BKB^T+CKC^T=S_A+S_B+S_C,
\end{align*}
we can obtain the identity
\begin{align*}
    |XA|^2+|XB|^2+|XC|^2=|GA|^2+|GB|^2+|GC|^2+3|GX|^2.
\end{align*}
Therefore,
\begin{align*}
    |XA|^2+|XB|^2+|XC|^2\geq|GA|^2+|GB|^2+|GC|^2.
\end{align*}
\end{proof}
\noindent The altitude of the triangle $\triangle ABC$ is defined as the line that passes through a vertex and is perpendicular to the opposite side. Therefore, the equations of the altitudes through points $A$, $B$, and $C$ are, respectively,
\begin{align*}
    &\left(B-C\right)K\left(X-A\right)^T=0,\\
    &\left(C-A\right)K\left(X-B\right)^T=0,\\
    &\left(A-B\right)K\left(X-C\right)^T=0.
\end{align*}
The identity
\begin{align*}
    \left(B-C\right)K\left(X-A\right)^T+
    \left(C-A\right)K\left(X-B\right)^T+
    \left(A-B\right)K\left(X-C\right)^T=0,
\end{align*}
shows that the three altitudes of a triangle intersect at a point, denoted as the orthocenter $H$ ($X(4)$ in ETC). 
Let $O=\frac12\left(3G-H\right)$, then
\begin{align*}
    \left(B-C\right)K\left(2O-B-C\right)^T
    =&\left(B-C\right)K\left(3G-H-B-C\right)^T\\
    =&\left(B-C\right)K\left(A-H\right)^T=0.
\end{align*}
By symmetry, we have
\begin{align*}
   &\left(C-A\right)K\left(2O-C-A\right)^T=0,\\
   &\left(A-B\right)K\left(2O-A-B\right)^T=0.
\end{align*}
Note that three perpendicular bisectors of the triangle intersect at the point $O$, which is the circumcenter of the triangle $\triangle ABC$ ($X(3)$ in ETC). Let the circumradius be $R$, according to the sine rule, we have $R = \frac{abc}{4S}$. To summarize the above results, we obtain:
\begin{theorem}[Euler's theorem]
    \begin{align}
    3G=2O+H,
    \label{triangle ABC Euler}
\end{align}
\end{theorem}
The calculation yields two important properties of $H$ and $O$ as follows:
\begin{theorem}
\begin{align}
    &HKX^T=HKH^T,\label{H kernel property}\\
    &OKO^T+HKH^T=R^2.
    \label{OKO-HKH property}
\end{align}
\end{theorem}
\begin{proof}
    Since $\left(B-C\right)K\left(H-A\right)^T=0$, then $BKH^T=CKH^T$. By symmetry, we obtain
\begin{align*}
    AKH^T=BKH^T=CKH^T.
\end{align*}
Therefore,
\begin{align*}
    HKX^T=\left(x+y+z\right)HKA^T=HKH^T.
\end{align*}
Using Euler's theorem and the equation (\ref{H kernel property}), $R$ can be calculated as follows:
    \begin{align*}
    R^2&=\left(O-A\right)K\left(O-A\right)^T=OKO^T-2OKA^T+S_A\\
    &=OKO^T-(3G-H)KA^T+S_A=OKO^T+HKH^T.
\end{align*}
\end{proof}
\noindent The equation (\ref{H kernel property}) is called the H kernel property, which is widely employed in the present paper. For instance, for any two points $X,\,Y$, we have
\begin{align*}
    \left(X-Y\right)KH^T=0.
\end{align*}
Using equations (\ref{H kernel property}) and (\ref{OKO-HKH property}), two equations of the circle are acquired as follows:
\begin{theorem}
The equation of the circumcircle of the triangle $\triangle ABC$ is
\begin{align}
    XKX^T-3GKX^T=0.
    \label{ABC cricumscribe 1}
\end{align}
\end{theorem}
\begin{proof}
    \begin{align*}
    &|XO|^2-R^2\\
    =&XKX^T-2OKX^T+OKO^T-R^2\\
    =&XKX^T-3GKX^T+HKH^T+OKO^T-R^2\\
    =&XKX^T-3GKX^T.
\end{align*}
\end{proof}
\begin{theorem}
The equation of the nine-point circle of the triangle $\triangle ABC$ is
\begin{align}
    XKX^T-\frac 32GKX^T=0.
    \label{ABC nine-point circle}
\end{align}
\end{theorem}
\begin{proof}
Consider the circle passing through points $\frac12 \left(A+H\right),\,\frac12 \left(B+H\right),\,\frac12 \left(C+H\right)$. Obviously, it is similar to the circumcircle with respect to the orthocenter $H$ (with a similarity ratio of $1:2$). That is, if a point $X$ lies on this circle, then it $2X - H$ lies on the circumcircle. Based on this, the equation of the circle can be obtained as
\begin{align*}
    XKX^T-\frac 32GKX^T=0.
\end{align*}
Let $M$ be the midpoint of side $BC$, obviously,
\begin{align*}
    MKM^T=\frac 14a^2=\frac 32GKM^T.
\end{align*}
By symmetry, the circle passes through the midpoint of three sides of the triangle $ABC$. Leaving the foot of $H$ to the side $BC$ as $H_A$, it is easy to acquire $H_AKA^T=0$ and $\left(H_A-H\right)K\left(H_A-B\right)^T=0$. According to H kernel property, we obtained
\begin{align*}
    H_AKH_A^T=H_AKB^T=H_AKC^T.
\end{align*}
Therefore,
\begin{align*}
    H_AKH_A^T=\frac 12\left(B+C\right)KH_A^T=\frac 32GKH_A^T.
\end{align*}
By symmetry, the feet of $H$ to the sides $CA$ and $AB$ both lies on the circle. The center of nine-point circle is denoted as $N$ ($X(5)$ in ETC), and it is clear that $N = \frac{1}{2}\left(O + H\right)$, with a radius of $\frac{1}{2}R$. 
\end{proof}
\noindent Considering the point $I$ inside the triangle $\triangle ABC$ that is equidistant from the three sides of the triangle $ABC$, we have
\begin{align*}
\left[B;C;I\right]:\left[C;A;I\right]:
\left[A;B;I\right]=a:b:c.
\end{align*}
Therefore, $I$ can be expressed as
\begin{align*}
    I=\frac{1}{2p}\left(aA+bB+cC\right),
\end{align*}
where $p=\frac 12\left(a+b+c\right)$ is the semiperimeter. Therefore, the three internal angle bisectors of the triangle $\triangle ABC$ intersect at $I$ ($X(1)$ in the ETC), which is also the center of the incircle of the triangle $\triangle ABC$. Let the radius of this circle be $r$, Using the area property, we have $r = \frac{S}{p}$.

\noindent Similarly, $I_A$ is denoted as
\begin{align*}
    &I_A=\frac{1}{2\left(p-a\right)}\left(-aA+bB+cC\right).
\end{align*}
The distances $I_A$ to the sides of the triangle $\triangle ABC$ are also equal. $I_A$ is the center of the excircle opposite the vertices $A$ of the triangle $\triangle ABC$. Let the radius of this circle be $r_a$, we have $r_a=\frac{S}{p-a}$. Simple calculations yield:
\begin{align}
    & IKI^T=2r^2,&&3GKI^T=2Rr+2r^2,  
    \label{IKI property}
\end{align}
and
\begin{align}
       &I_AKI_A^T=2r_a^2
    &&3GKI_A^T=-2Rr_a+2r_a^2. 
    \label{IAKIA property}
\end{align}
Using the equations (\ref{IKI property}) and (\ref{IAKIA property}), the following conclusions are obtained:
\begin{theorem}
The equation of the incircle of the triangle $\triangle ABC$ is
\begin{align*}
    XKX^T-2IKX^T+\frac 12IKI^T=0.
\end{align*}
The equation of the excircle opposite to vertex $A$ is
\begin{align*}
    XKX^T-2I_AKX^T+\frac 12I_AKI_A^T=0.
\end{align*}
\end{theorem}
\noindent Using the equations (\ref{H kernel property}), (\ref{OKO-HKH property}), and (\ref{IKI property}), two famous theorems about the distances between the centers of the triangle can be proved concisely:
\begin{theorem}[Euler's formula]
    \begin{align*}
        |OI|^2=R^2-2Rr.
    \end{align*}
\end{theorem}
\begin{proof}
    \begin{align*}
    |OI|^2&=OKO^T-2OKI^T+IKI^T \\
    &=OKO^T-\left(3G-H\right)KI^T+IKI^T \\
    &=R^2-\left(2Rr+2r^2\right)+2r^2 \\
    &=R^2-2Rr,
\end{align*}
\end{proof}
At the same time, the Euler's inequality $R \geq 2r$ can be obtained. Similarly (using equation (\ref{IAKIA property})), we have
\begin{align*}
    |OI_A|^2=R^2+2Rr_a.
\end{align*}
\begin{theorem}[Feuerbach's theorem]
The nine-point circle is tangent to the incircle and externally tangent to the three excircles, i.e.,
\begin{align*}
    &|NI|=\frac12 R-r,
    &&|NI_A|=\frac 12R+r_a.
\end{align*}
\end{theorem}
\begin{proof}
    \begin{align*}
    |NI|^2&=\left(\frac{O+H}{2}-I\right)K\left(\frac{O+H}{2}-I\right)^T \\
    &=\frac{1}{4}\left(OKO^T+4IKI^T-4OKI^T-HKH^T\right) \\
    &=\frac{1}{4}\left(2IKI^T+
    2\left(O-I\right)K\left(O-I\right)^T-\left(OKO^T+HKH^T\right)\right) \\
    &=\left(\frac 12 R-r\right)^2.
\end{align*}
Similarly, $|NI_A|^2=\left(\frac 12 R+r_a\right)^2$
\end{proof}
\section{Conclusion}
In this study, metric matrix is introduced to yield a matrix representation of the inner product in barycentric coordinates, and several related corollaries have been proved, as outlined in Section \ref{sec:bary}. Building on these results, the fundamental properties of triangle geometry have been explored in Section \ref{sec:triangle}.

The novel results obtained through this framework demonstrate that the methods developed herein not only improve computational efficiency but also provide deeper geometric insight, thereby underscoring the utility of barycentric coordinates as both a theoretical foundation and a practical tool for advanced research in geometry.

Looking ahead, future work will seek to extend this framework to cover a broader range of geometric topics, including lines, circles, conic sections, quadrilaterals and geometric inequalities. Additionally, we anticipate that the system introduced in this paper will not only contribute to the discovery of new results in triangle geometry but also hold significant potential for applications in related fields, such as computer graphics and computational geometry.


\begin{thebibliography}{20}
\bibitem{capitan2015barycentric}
Francisco Javier~Garc{\i}a Capit{\'a}n.
\newblock Barycentric coordinates.
\newblock {\em International Journal of Computer Discovered Mathematics}, pages 32--48, 2015.

\bibitem{grozdev2016barycentric}
Sava Grozdev and Deko Dekov.
\newblock Barycentric coordinates: formula sheet.
\newblock {\em International Journal of Computer Discovered Mathematics}, 1(2):75--82, 2016.

\bibitem{johnson2013advanced}
Roger~A Johnson.
\newblock {\em Advanced euclidean geometry}.
\newblock Courier Corporation, 2013.

\bibitem{kimberling_etc}
C.~Kimberling.
\newblock Encyclopedia of triangle centers-etc.
\newblock {\em faculty.evansville. edu/ck6/encyclopedia/ETC.html}


\bibitem{schindler2012barycentric}
Max Schindler and Evan Chen.
\newblock Barycentric coordinates in olympiad geometry.
\newblock {\em Olympiad Articles}, 1, 2012.

\bibitem{volenec2003metrical}
Vladimir Volenec.
\newblock Metrical relations in barycentric coordinates.
\newblock {\em Mathematical communications}, 8(1):55--68, 2003.

\bibitem{volenec2015baricentricke}
Vladimir Volenec.
\newblock Baricentricke koordinate 2 metricka svojstva.
\newblock {\em Osjecki Matematicki List}, 15(2), 2015.

\bibitem{yiu2001introduction}
Paul Yiu.
\newblock {\em Introduction to the Geometry of the Triangle}.
\newblock Florida Atlantic University Lecture Notes, 2001.
\bibitem{zolotov2022scalar}
Vladimir Zolotov.
\newblock Scalar product and distance in barycentric coordinates.
\newblock {\em arXiv preprint arXiv:2212.11712}, 2022.
\end{thebibliography}
\end{document}